\documentclass[11pt]{article}

\usepackage{amsmath}
\usepackage{amsfonts}
\usepackage{amssymb}
\usepackage{amsthm}
\usepackage{accents}
\usepackage[english]{babel}
\usepackage{mathrsfs}
\usepackage[utf8]{inputenc}

\usepackage{stmaryrd} 

\usepackage{graphicx}
\usepackage[percent]{overpic}
\usepackage{hyperref}
\usepackage{varioref}
\usepackage[mathscr]{euscript}
\usepackage{caption}
\usepackage{subcaption} 
\usepackage{calrsfs}
\usepackage{enumerate,enumitem}
\usepackage{color,tikz}
\usepackage{tikz-cd}
\usepackage{bbold}
\usetikzlibrary{arrows}
\usetikzlibrary{positioning}
\usetikzlibrary{shapes.multipart}

\usepackage[utf8]{inputenc}
\usepackage[left=2cm,right=2cm,top=2cm,bottom=2cm]{geometry}

\begin{document}
\newcommand{\cblue}{\color{blue}}
\newcommand{\cgreen}{\color{green}}
\newcommand{\cbred}{\bf\color{red}}
\newcommand{\cred}{\color{red}}

\newcommand{\bm}{\bibitem}
\newcommand{\no}{\noindent}

\newcommand{\bea}{\begin{eqnarray}}
\newcommand{\bes}{\begin{subequations}}
\newcommand{\ees}{\end{subequations}}
\newcommand{\bgt}{\begin{gather}}

\newcommand{\egt}{\begin{gather}}
\newcommand{\eea}{\end{eqnarray}}
\newcommand{\beaa}{\begin{eqnarray*}}
\newcommand{\eeaa}{\end{eqnarray*}}

\renewcommand{\baselinestretch}{1.1}

\def \D{\mathbb{D}}
\def \E{\mathbb{E}}
\def \F{\mathbb{F}}
\def \L{\mathbb{L}}
\def \P{\mathbb{P}}
\def \Q{\mathbb{Q}}
\def \R{\mathbb{R}}
\def \S{\mathbb{S}}
\def \M{\mathbb{M}}
\def \N{\mathbb{N}}
\def \C{\mathbb{C}}

\def\Ac{{\cal A}}
\def\Bc{{\cal B}}
\def\Cc{{\cal C}}
\def\Dc{{\cal D}}
\def\Ec{{\cal E}}
\def\Fc{{\cal F}}
\def\Gc{{\cal G}}
\def\Hc{{\cal H}}
\def\Ic{{\cal I}}
\def\Jc{{\cal J}}
\def\Kc{{\cal K}}
\def\Lc{{\cal L}}
\def\Mc{{\cal M}}
\def\Nc{{\cal N}}
\def\Oc{{\cal O}}
\def\Pc{{\cal P}}
\def\Qc{{\cal Q}}
\def\Rc{{\cal R}}
\def\Sc{{\cal S}}
\def\Tc{{\cal T}}
\def\Uc{{\cal U}}
\def\Vc{{\cal V}}
\def\Wc{{\cal W}}
\def\Xc{{\cal X}}
\def\Yc{{\cal Y}}
\def\Zc{{\cal Z}}

\def\Lb{\overline L}
\def\Pb{\overline{\P}}
\def\Qb{{\bar \Q}}
\def\Rb{{\bar \R}}
\def\Rcb{{\bar{\Rc}}}
\def\Db{\bar{D}}
\def\Dbb{\bar{\D}}
\def\Dcb{\bar{\Dc}}
\def\Cb{\overline{C}}
\def\Cu{\underline{C}}
\def\Ob{\overline{O}}
\def\Ocb{\overline{\Oc}}
\def\Bh{\widehat{B}}
\def\Lcb{\overline{\Lc}}

\def\Rct{\tilde{\Rc}}
\def\Et{\tilde{\E}}
\def\Eb{\overline{\E}}

\def\Nt{\tilde{N}}
\def\Bt{\tilde{B}}
\def\Kct{\widetilde{\Kc}}
\def\Omt{\tilde{\Om}}
\def\Fct{\tilde{\Fc}}
\def\Pt{\tilde{\P}}
\def\pt{\widetilde{p}}
\def\Yt{\tilde{Y}}
\def\Zt{\tilde{Z}}

\def\Ecu{\overline{\Ec}}
\def\Ecb{\underline{\Ec}}
\def\Acu{\overline{\Ac}}
\def\Acb{\underline{\Ac}}
\def\Ucu{\overline{\Uc}}
\def\Ucb{\underline{\Uc}}
\def\Lch{\widehat{\Lc}}

\def\taub{\bar \tau}

\def\Thetau{\bar{\Theta}}
\def\psiu{\overline{\psi}}
\def\psib{\underline{\psi}}

\def\Mbf{{\bf M}}
\def \I{{\bf I}}
\def \au{\overline{\alpha}}
\def \al{\underline{\alpha}}
\def \sigmal{\underline{\sigma}}
\def \a{\alpha}

\def \Om{\Omega}
\def \om{\omega}
\def \Omb{\overline{\Omega}}
\def \omb{\overline{\om}}
\def \omh{\hat{\om}}
\def \tauh{\hat{\tau}}
\def \omt{\tilde{\om}}
\def \eps{\epsilon}
\def \xb{\mathbf{x}}
\def \xbh{\hat{\xb}}
\def \0{\mathbf{0}}
\def \H{\mathbb{H}}
\def \Xb{\overline{X}}

\def \Lambdab{\overline{\Lambda}}
\def \ab{\bar{a}}

\def \Sg{\Sigma}
\def \Fcb{\overline{{\cal F}}}
\def \Fbb{\overline{\mathbb{F}}}

\def \Pcb{\overline{\Pc}}
\def \Kcb{\overline{\Kc}}
\def \psih{\widehat{\psi}}

\def \vp{\varphi}
\def \x{\times}
\def \sigmah{\widehat \sigma}
\def \yr{\mathrm{y}}

\def \Omh{\widehat \Om}
\def \Fch{\widehat \Fc}
\def \Fh{\widehat \F}
\def \Ph{\widehat \P}
\def \Xh{\widehat X}
\def \Wh{\widehat W}
\def \T{\mathbb{T}}
\def \Z{\mathbb{Z}}
\def \ph{\widehat p}
\def \Vh{\widehat V}
\def \muh{\widehat \mu}

\newcommand{\var}{{\rm Var}}
\newcommand{\rmi}{{\rm (i)$\>\>$}}

\newcommand{\rmii}{{\rm (ii)$\>\>$}}
\newcommand{\rmiii}{{\rm (iii)$\>\,    \,$}}
\newcommand{\rmiv}{{\rm (iv)$\>\>$}}
\newcommand{\rmv}{{\rm (v)$\>\>$}}

\newcommand{\rma}{{\rm a)$\>\>$}}
\newcommand{\rmb}{{\rm b)$\>\>$}}
\newcommand{\rmc}{{\rm c)$\>\>$}}
\newcommand{\rmd}{{\rm d)$\>\>$}}

\def\1{\mathbf{1}}
\def\xr{{\rm x}}
\def\vr{{\rm v}} 
\def\ti{t_{i}^{n}}
\def\tip{t_{i+1}^{n}}
\def\bru#1{{\color{red}{#1}}}
\def\dr{{\mathfrak d}}
\def\dr{{d}}
\def\CT{C([0,T])}
\def\CTt{\widetilde C([0,T])}
\def\DT{D([0,T])}     
\def\TD{[0,T]\x\DT}
\def\TC{[0,T]\x\CT}
\def\v#1{_{\wedge #1}}
 \def\Ninfty#1{\|#1\|}    
 \def\Cb{{\mathbb C}}      
 \def\Href#1{{\rm (H\ref{#1})}}
  \def\HHref#1{{\rm (H\ref{#1}$^{*}$)}}
  \def\wr{{\rm w}}
  \def\vs#1{\vspace{2mm}}
  \def\Lb{{\mathbf L}}
  \def\Pas{\P-{\rm a.s.}}
  \def\bru#1{{\color{red} #1}}
    \def\brub#1{{\color{purple} #1}}
\def\greg#1{{\color{red}#1}}
\def\red#1{{\color{red}#1}}
\def\blue#1{{\color{blue}#1}}

\title{Itô-Dupire's formula for $\C^{0,1}$-functionals of càdlàg weak Dirichlet processes}
\author{Bruno Bouchard\footnote{CEREMADE, Universit\'e Paris-Dauphine, PSL, CNRS.  bouchard@ceremade.dauphine.fr. } 
, Maximilien Vallet\footnote{CEREMADE, Universit\'e Paris-Dauphine, PSL, CNRS.  vallet@ceremade.dauphine.fr. } 
}
\maketitle 
$\newtheorem{Theorem}{Theorem}[section]$
$\newtheorem{Lemma}[Theorem]{Lemma}$
$\newtheorem{Proposition}[Theorem]{Proposition}$
$\newtheorem{Remark}[Theorem]{Remark}$
$\newtheorem{Definition}[Theorem]{Definition}$
$\newtheorem*{Assumption}{Assumption}$

\begin{abstract} 
	We extend to   c\`adl\`ag weak Dirichlet processes the $\mathbb{C}^{0,1}$-functional It\^{o}-Dupire's formula of Bouchard, Loeper and Tan (2021). In particular, we provide sufficient conditions under which a $\mathbb{C}^{0,1}$-functional transformation of a special weak Dirichlet process remains a special weak Dirichlet process. As opposed to  Bandini and Russo (2018) who considered the Markovian setting, our approach is not based on the approximation of the functional by smooth ones, which turns out not to be available in the path-dependent case. We simply use a small-jumps cutting argument.  
\end{abstract}

\section{Introduction}
Let $X=X_0+M+A$ be a  càdlàg  semimartingale where $M=M^c+M^d$ is a local martingale and $A$ is adapted and of bounded variations. Let $\mu^{X}$ denotes its jump measure and $\nu^{X}$ its compensator. Then, given a $C^{1,2}$ function $F:[0,T]\x\mathbb{R}^d\to\mathbb{R}$, the  Itô's formula ensures that $(F(t,X_t))_{t\in\left[ 0,T\right] }$ is a  semimartingale with decomposition
\begin{align*} 
F(t,X_t)=F(0,X_0)+&\int_0^t\nabla_x F(s,X_{s-}) dM_s\notag\\
&+\int_{\left] 0,t\right] \times\mathbb{R}^d}(F(s,X_{s-}+x)-F(s,X_{s-})-x\cdot \nabla_x F(s,X_{s-}))(\mu^X-\nu^{X})(ds,dx)\\
&+\Gamma^F_t    
\end{align*}
where  
\begin{align*}
\Gamma^F_t&=\int_0^t\partial_tF(s,X_{s}) ds+\int_0^t\nabla_x F(s,X_{s-}) dA_s+\dfrac{1}{2}\sum_{1\leq i,j\leq d}\int_0^t\nabla^2_{x^{i}x^{j}}F(s,X_{s-}) d\left[ X^i,X^j\right]^c_s.
\end{align*}
If we assume furthermore that $F(\cdot,X_{\cdot})$ is a local martingale, then  $\Gamma^{F}\equiv 0$, and 
this formula only uses the first derivative in space of $F$ and should be valid even if $F$ is only  ${C}^{0,1}$. In the Markovian setting, we know from \cite{bandini2017weak,coquet2006natural} that it is indeed true for càdlàg weak Dirichlet processes, even when $F(\cdot,X_{\cdot})$ is not a local martingale, in which case $\Gamma^{F}$ turns out to be an orthogonal process, which is even predictable if $X$ is special. In \cite{bouchard2021ac}, the authors provide an extension of this result to the path-dependent case under the condition that $X$ has continuous path. Naturally, it uses  the notion of Dupire's derivative, see \cite{Dupire2009FunctionalIC,cont2013functional}. 

Such a decomposition appears to be a powerful tool in particular for using verification arguments in optimal control problems, for which obtaining a $C^{1,2}$-type regularity for the value function may be difficult, if even true. The situation is worse when it comes to considering path-dependent problems, for which classical derivatives have to be replaced by the notion of Dupire's derivatives, whose existence and regularity are difficult to obtain.   Versions of the above formula were actually already applied successfully in   \cite{BT19,bouchard2021quasi,bouchard2021ac} in the context of risk hedging, under model uncertainty or in markets with price impacts. See    \cite{bouchard2021approximate} for an application to  BSDEs, or for a class of so-called $\pi$-approximate viscosity solutions of fully non-linear parabolic path-dependent PDEs, for which $C^{0,1}$-regularity in the sense of Dupire can be obtained. 
   
When, as in \cite{bouchard2021ac}, $X$ has continuous path, then it is immediate to conclude that $\Gamma^{F}$ is predictable. Things are a priori more complex if $X$ has jumps. In this case, the above decomposition into a weak Dirichlet process is not unique, and the orthogonal process $\Gamma^{F}$ can contain a purely discontinuous martingale part, which makes the above decomposition useless for verification arguments. In the Markovian setting,  \cite{bandini2017weak} uses an approximation argument on $F$ to show that it can actually be chosen to be predictable if $X$ is special. Namely, they construct a sequence of predictable processes $(\Gamma^{F_{n}})_{n\ge 1}$ obtained by applying It\^{o}'s formula to  smooth approximations $(F_{n})_{n\ge 1}$ of $F$, and then show that $(\Gamma^{F_{n}})_{n\ge 1}$ converges to $\Gamma^{F}$. 
This argument could not be extended so far to the case where $F$ is path-dependent. The main reason is that the vertical and horizontal Dupire's derivatives do not commute, which renders the construction of smooth (in the sense of Dupire) approximations a completely open problem, see e.g.~\cite{saporito2018functional}. 

In this paper, we follow a different and actually simpler route. First, we observe that the decomposition can easily be deduced from   \cite{bouchard2021ac} when $X$ does not have small jumps. Then, we just approximate $X$ by removing its small jumps, and passing to the limit. 

The rest of the paper is organized as follows. We first recall  usefull results of the functionnal Itô calculus and Itô calculus via regularization. Then, we state and demonstrate our version of the functionnal Itô's formula for càdlàg special weak Dirichlet processes. We conclude with a typical exemple of application. Some (essentially known) technical results are collected in the Appendix for completeness.

\section{Notations and definitions}

All over  this paper, we fix a time horizon  $T>0$ and let $(\Omega,\mathcal{F},(\mathcal{F}_t)_{t\in \left[ 0,T\right] },\P)$ be a stochastic basis, $i.e.$ a filtered probability space such that the filtration $(\mathcal{F}_t)_{t\in \left[ 0,T\right] }$ is right continuous.

\subsection{Skorokhod space and path-dependent functionnals}
Let $D(\left[ 0,T\right] )$ be the set of càdlàg paths on $\left[  0,T\right] $ taking values in $ \mathbb{R}^d$ and $\Theta :=\left[  0,T\right]\times D(\left[ 0,T\right] )$. 
For $(t,\xr)\in\Theta $, we define the (optional-) stopped path $\xr_{t\wedge}\in D(\left[ 0,T\right] )$ by $\xr_{t\wedge}:=\xr\mathbb{1}_{ \left[ 0,t\right[  }+\xr_t\mathbb{1}_{\left[ t,T\right]  }$ and its predictable version $\xr^{-}_{t\wedge}\in D(\left[ 0,T\right] )$ by $\xr^{-}_{t\wedge}:=\xr\mathbb{1}_{ \left[ 0,t\right[  }+\xr_{t-}\mathbb{1}_{\left[ t,T\right]  }$.  
For $(t,\xr)\in\Theta $ and $y\in\mathbb{R}^d$, we also define the trajectory $\xr\oplus_ty$ by $\xr \mathbb{1}_{\left[ 0,t\right[  }+(\xr_t+y)\mathbb{1}_{ \left[ t,T\right] }$ and the trajectory  $\xr\boxplus_ty$ by $\xr \mathbb{1}_{ \left[ 0,t\right[  }+y\mathbb{1}_{\left[ t,T\right]  }$. 

We define on $\Theta$ the pseudo-distance   $d_\Theta((t,\xr),(t',\xr'))=\lvert t'-t \rvert + \lvert\!\lvert \xr '_{t'\wedge}-\xr_{t\wedge}\rvert\!\rvert$, where $\vert\!\lvert\cdot\rvert\!\rvert$ denotes the uniform norm on $D(\left[ 0,T\right] )$. Considering the quotient space $(\Theta,\sim)$  defined by $(t,\xr)\sim(t',\xr')$ whenever $t=t'$ and $\xr_{t\wedge }=\xr'_{t\wedge }$,    $(\Theta,d_\Theta)$ is a complete metric space.  

We say that $F:\Theta\rightarrow\mathbb{R}$ is non-anticipative if $F(t,\xr)=F(t,\xr_{t\wedge})$    $\forall(t,\xr)\in\Theta $. 
A non anticipative function $F:\Theta\rightarrow\mathbb{R}$ is said continuous if it is continuous for $(\Theta,d_\Theta)$. The set of continuous non-anticipative maps on $\Theta$ will be denoted $\mathbb{C}(\Theta)$.
We say that $F$ is locally uniformly continuous if, for all $K>0$, there exists a modulus of continuity $\delta_K(F,\cdot)$ ($i.e.$ a non-negative and non-decreasing function defined on $\mathbb{R}_+$ that is continuous at $0$ and vanishes at $0$) such that 
\begin{equation}\label{def_module_continuit}
\lvert F(t,\xr)-F(t,\xr')\rvert\leq\delta_K(F,d_\Theta((t,\xr),(t',\xr')))
\end{equation}
 for all $(t,\xr),(t',\xr')\in\Theta$ with $\| \xr\|\vee\| \xr'\|\leq K$.\\
A functionnal $F:\Theta\rightarrow\mathbb{R}$ is said to be locally bounded if 
\begin{equation*}
\sup_{t\in\left[ 0,T\right],~\lvert\!\lvert{\xr}\rvert\!\rvert\leq{K}}\lvert F(t,\xr)\rvert <+\infty,~~\forall{K}\in\R_+~.
\end{equation*}\\
We denote by $\C_{loc}^{u,b}(\Theta)$ the set of non anticipative,  locally uniformly continuous and locally bounded functionnals.

We can now define the notion of differentiability for path-dependent functionals following the one introduced by Dupire in \cite{Dupire2009FunctionalIC}:  a non-anticipative function $F:\Theta\rightarrow\mathbb{R}$ is said to be vertically differentiable at $(t,\xr)\in\Theta $ if $y\in\R^d\mapsto F(t,\xr\oplus_ty)$ is differentiable at 0.  In this case,  we denote by $\nabla_\xr F(t,\xr)$  this differential. We denote by  $\mathbb{C}^{0,1}(\Theta)$ the collection of non-anticipative functions $F$ such that $\nabla_\xr F$ is well-defined and continuous on $\Theta$.

In this paper, for all path-dependent functionnal $F$ defined on $\Theta$ and $(t,\xr)\in\Theta$, we will use the notations 
$$
F_t(\xr):=F(t,\xr)\;\mbox{ and } \; F_{t}(\xr^{-}):=F_{t}(\xr^{-}_{t\wedge}).
$$
\subsubsection{It\^o's calculus via regularization and weak Dirichlet processes}

	Let us recall here some definitions and facts on the It\^o calculus via regularization developped by Russo and Vallois \cite{russo1993forward, russo1995generalized, russo2007elements}. 
	See also Bandini and Russo \cite{bandini2017weak} for the case of c\`adl\`ag processes.  For the rest of the paper u.c.p. means uniform convergence in probability.

	\begin{Definition} \label{def:integral}
		$\mathrm{(i)}$ Let $X$ be a real valued {c\`adl\`ag} process, and $H$ be a process with paths in $L^1([0,T])$ a.s. The forward integral of $H$ w.r.t. $X$ is defined by
		$$
			\int_0^t H_s ~d^-X_s 
			~:=~ 
			\lim_{\eps \searrow 0} \frac1\eps \int_0^t H_s \big( X_{(s+\eps) \wedge t} - X_s \big) ds,
			~~t \ge 0,
		$$
		whenever the limit exists in the sense of u.c.p.\\
		We naturally extend the definition of the forward integral for two $\R^d$-valued processes $X$ and $H$ such that $X^i$ is càdlàg and $H^i$ has paths in $L^1([0,T])$ for all $i=1\ldots d$ by 
		$$
		    \int_0^t H_s ~d^-X_s=\sum_{i=1}^d\int_0^t H^i_s ~d^-X^i_s	,	    
		~~t \ge 0,
		$$
		whenever all those integrals exist.\\
		\vspace{0.5em}
		
		\noindent $\mathrm{(ii)}$ 
		Let $X$ and $Y$ be two real valued {c\`adl\`ag} processes. The  quadratic {co}variation $[X,Y]$ is defined by 
		$$
			[X,Y]_{t}
			~:=~ 
			\lim_{\eps\searrow 0} \frac1\eps \int_{0}^{t} (X_{(s+\eps)\wedge t}-X_{s})(Y_{(s+\eps)\wedge t}-Y_{s})ds,
			~~t\ge 0,
		$$ 
		whenever the limit exists in the sense  of u.c.p.\\
		In the following, we will use the notation $\left[ X,Y\right] _{\eps,t}^{ucp}:=\frac1\eps \int_{0}^{t} (X_{(s+\eps)\wedge t}-X_{s})(Y_{(s+\eps)\wedge t}-Y_{s})ds$.\\
		We naturally define the quadratic {co}variation matrix $([X,Y]^{i,j})_{1\leq i,j\leq d}$ for two $\R^d$-valued càdlàg processes $X$ and $Y$ by, for all $1\leq i,j\leq d$,
		$$
		    [X,Y]^{i,j}_{t}=[X^i,Y^j]_{t},~~t\ge 0,
		$$ 
		whenever $[X^i,Y^j]$ is well defined for all $1\leq i,j\leq d$.
		\vspace{0.5em}
		
		\noindent $\mathrm{(iii)}$ We say that a $\R^d$-valued {c\`adl\`ag} process $X$ has finite quadratic variation,
		if its quadratic variation, defined by $[X]:=[X,X]$, exists and is finite a.s.
	\end{Definition}

	\begin{Remark}\label{consistence_bracket}
		When $X$ is a {(c\`adl\`ag)} semimartingale and $H$ is a  c\`adl\`ag adapted process,   $\int_0^t H_s ~d^-X_s$ coincides with the usual It\^o's integral $\int_0^t H_{s-} dX_s$.
	 	When $X$ and $Y$ are two semimartingales,   $[X, Y]$ coincides with the usual bracket.
 	\end{Remark}

	\begin{Definition} \label{def:weakDirichlet}
		$\mathrm{(i)}$ We say that an adapted process $A$ is orthogonal if $[A, N] = 0$ for any continuous local martingale $N$.
		
		\vspace{0.5em}
		
		\noindent $\mathrm{(ii)}$ An adapted   process $X$ is called a  (resp.~special) weak Dirichlet process if it has a decomposition of the form 
		 $X = X_0 + M +A$, where $M$ is a local martingale and $A$ is an (resp.~predictable) orthogonal process, such that $M_0 = A_0 = 0$.
		 	\end{Definition}
	
	\begin{Remark}\label{rem: weak dirichlet}
	     $\mathrm{(i)}$ An adapted  process with finite variation is orthogonal.
		Consequently,  a   semimartingale is in particular a weak Dirichlet process.

		\vspace{0.5em}
		
		 \noindent$\mathrm{(ii)}$ Any purely discontinuous local martingale is orthogonal by Remark \ref{consistence_bracket}.

		\vspace{0.5em}
		
		\noindent $\mathrm{(iv)}$ An orthogonal process has not necessarily finite  variations. For example,  any deterministic process (with possibly infinite variation) is orthogonal.

		\vspace{0.5em}
		
		\noindent $\mathrm{(iv)}$	
		The decomposition  $X= X_0 + M+A$ for a càdlàg weak Dirichlet process $X$ is not unique in general. Indeed, we can always displace a purely discontinuous martingale part in the orthogonal part. However, this decomposition is unique if $X$ is   special.  
 	\end{Remark}

\section{The It\^{o}-Dupire's formula for $\mathbb{C}^{0,1}$-functionals}

In \cite{bouchard2021ac}, the authors require an assumption relating the regularity of the path of $X$ and of the functional $F$, Assumption (A) below. When $X$ has continuous path, it turns out be equivalent to the decomposition \eqref{eq: Ito thm} below. In our setting, we shall apply it to an approximation of $X$ obtained by removing its small jumps, see Remark \ref{rem: applique A a Xn} below. 

\begin{Assumption}[A]
Let $F:\Theta\to \R$ be a non-anticipative functional and  $Y$ be a càdlàg process. We say that the couple $(F,Y)$ satisfies Assumption $(A)$ if
\begin{equation}\label{iff_Ito}
\dfrac{1}{\epsilon}\int_0^\cdot(F_{s+\epsilon}(Y)-F_{s+\epsilon}(Y_{s\wedge}\boxplus_{s+\epsilon}Y_{s+\epsilon}))(N_{s+\epsilon}-N_s)ds\underset{\epsilon\rightarrow{0}}{\longrightarrow}0~~u.c.p. 
\end{equation}
for all continuous martingale $N$.
\end{Assumption}
\begin{Remark}\label{suff_Ito} First note that the left-hand side of \eqref{iff_Ito} is always $0$ when $F$ is Markov, i.e.~$F(t,\xr)=F(t,\xr')$ whenever $\xr_{t}=\xr'_{t}$. Second, the above also holds if we assume that, for all $\xr\in \DT$, $s\in [0,T]$ and $\eps \in [0,T-s]$, 
		$$
			\big|F_{s+\eps}(\xr)- F_{s+\eps}(\xr_{s \wedge }\boxplus_{s +\eps} \xr_{s+\eps}) \big|
			~\le~ 
			\int_{(s, s+\eps)} \phi \big(\xr, |\xr_{u-} - \xr_{s }|\big)db_u (\xr),
		$$
		where  $\phi: \DT \x \R_+  \longrightarrow \R$ satisfies $ \sup_{|y|\le K} \phi(\xr,y )  <\infty$, $\lim_{y \searrow 0}  \phi(\xr,y ) = \phi(\xr, 0) =0$ for all $\xr \in \DT$ and $K > 0$,  
		and $b$ maps $\DT$ into the space $ {\rm BV}_{+}$ of non-decreasing bounded variation processes.
	This follows from the same arguments as in \cite[Proposition 2.11]{bouchard2021ac}. In particular, \eqref{iff_Ito} is satisfied  if $F$ is Fr\'echet differentiable in the sense of Clark \cite{clark1970representation}, see \cite[Example 2.12]{bouchard2021ac}.
		\end{Remark}
We are now ready to state our decomposition result. From now on, we fix a càdlàg special weak Dirichlet process 
\begin{equation}\label{eq: def X}
X=X_0+M+A.
\end{equation}
Here, $M=M^{c}+M^{d}$ where $M^{c}$ and $M^{d}$ denote its continuous and purely discontinuous martingale parts, and $A$ is a predictable orthogonal process (recall that the decomposition is unique in this case, see Remark \ref{rem: weak dirichlet}). We denote by $\mu^X$ the jump measure of $X$, and by $\nu^X$ its compensator.  If $\sum_{{0\leq s \leq T}}\lvert \Delta A_s \rvert<+\infty$ a.s., then one can define the continuous part of $X$ by   $$X^c=X-M^d-{\sum_{s\le \cdot}} \Delta A_s.$$

\begin{Theorem}\label{TIto} Let $X$ be as in \eqref{eq: def X} and assume that 
\begin{align}\label{eq: hyp jumps} 
[X]_{T} +\sum_{0\leq s \leq T}\lvert \Delta A_s \rvert<+\infty \;\mbox{ a.s.} 
\end{align}
Let $F\in\mathbb{C}^{0,1}(\Theta)$ be such that $F$ and $\nabla_{\xr}F$ are both in $\mathbb{C}_{loc}^{u,b}(\Theta)$ and such that $t\in [0,T]\mapsto\nabla_\xr F_{t}(X^-)$ admits right-limits $a.s$. Assume further that $(F,Z\oplus_\tau(X^c-X_\tau^c))$ satisfies Assumption $(A)$ for every càdlàg process $Z$ and stopping time $\tau$ such that $\tau\le T$ a.s.\\
Then, $(F_{t}(X))_{t\in{\left[ 0,T \right] }}$ is a special weak Dirichlet process with decomposition
\begin{align}\label{eq: Ito thm}
F_{t}(X)=F_{0}(X)+&\int_0^t\nabla_\xr F_{s}(X^-) dM_s\notag\\
&+\int_{\left] 0,t\right] \times\mathbb{R}^d}(F_s(X^-\oplus_s x)-F_s(X^-)-x\nabla_\xr F_{s}(X^-)) {(\mu^X-\nu^{X})}(ds,dx)\notag\\
&+\Gamma_t^F,~~\forall t\in\left[ 0,T\right],
\end{align}
where $\Gamma^F$ is an orthogonal and predictable process.
\end{Theorem}

Before to provide the proof of this result, let us make several comments. 

\begin{Remark}
All the terms in \eqref{eq: Ito thm} are well defined. In particular, 
\begin{align*}
&\int_{\left] 0,\cdot\right] \times\mathbb{R}^d}(F_s(X^-\oplus_s x)-F_s(X^-))\mathbb{1}_{\left\lbrace \lvert{x}\rvert\leq{1}\right\rbrace }(\mu^X-\nu^X)(ds,dx)\\ &\int_{\left] 0,\cdot\right] \times\mathbb{R}^d}x\nabla_\xr F_{s}(X^-)\mathbb{1}_{\left\lbrace \lvert{x}\rvert\leq{1}\right\rbrace }(\mu^X-\nu^X)(ds,dx)
\end{align*}
are purely discontinuous local martingales. See Lemma \ref{purely _discontinuous} below.
\end{Remark}

\begin{Remark}
If $X$ is {a semimartingale}, then  \eqref{eq: hyp jumps} holds.
\end{Remark}

\begin{Remark}\label{rem : def Zn} Let $(\epsilon_n)_{n\in\mathbb{N}}\subset (0,1)^{\N}$ be a decreasing  sequence of positive real numbers converging to $0$. Using \eqref{eq: hyp jumps}, we can define   $Z^n:=Y^n+\sum_{s\leq\cdot}\Delta A_s\mathbb{1}_{\lvert\Delta A_s\rvert< \eps_n}$ where $Y^n:=x\mathbb{1}_{\left\lbrace \lvert x\rvert< \eps_n\right\rbrace }\ast(\mu^X-\nu^X)$ is a purely discontinuous local martingale  (see \cite[Theorem ~II.2.34]{jacod2013limit}). Then,  $Z^n$ is an  orthogonal special semi-martingale with jumps not larger than $\epsilon_n$, namely $\lvert\Delta Z^n_t\rvert<\epsilon_n$ $\forall t\in \left[ 0,T\right]$ a.s., such that $X^n:=X-Z^n$ only has jumps larger than $\epsilon_n$.  Moreover, $\lvert\!\lvert Z^n\rvert\!\rvert+\left[  Z^n\right]_T  {\to}0$ a.s~as $n\to \infty$.  
\end{Remark}

\begin{Remark}\label{rem: applique A a Xn} For simplicity of exposition of our main result, we assumed that $(F,Z\oplus_\tau(X^c-X_\tau^c))$ satisfies Assumption $(A)$ for every càdlàg process $Z$ and stopping time $\tau$ such that $\tau\le T$ a.s.
In the proof, we shall actually only use the fact that $(F,X^{n}\oplus_\tau(X^c-X_\tau^c))$ satisfies Assumption $(A)$ for all stopping time $\tau$ corresponding to a jump time of $X^{n}$, for all $n\ge 1$.
\end{Remark}

\begin{proof}[Proof of Theorem \ref{TIto}]
{\rm 1.} The fact that the decomposition \eqref{eq: Ito thm} holds with $\Gamma^{F}$ orthogonal, but not necessarily predictable, follows from the same arguments as in \cite{bandini2017weak,bouchard2021ac}, see Proposition \ref{ItoC01standard} in the Appendix. 
We therefore just have to show that $\Gamma^F$ is predictable.

\noindent {\rm 2.} Let $(\epsilon_n,X^{n},Y^{n},Z^{n})_{n\in\mathbb{N}}$ be as in Remark \ref{rem : def Zn}.  
Fix $n\in\mathbb{N}$ and let $(\tau_k^n)_{k\in\mathbb{N}}$ be the sequence of stopping times corresponding to the jumps of $X$ larger or equal to $\epsilon_n$,  namely $\tau_0^n=0$ and $\tau_{k+1}^n=\inf\{ s> \tau_k^n$ s.t.~$\lvert\Delta X_s \rvert\geq \epsilon_n\}$.  These are the jump times of $X^{n}$. Then,  $K^n:=\min\{ k\in \N$ s.t.~$\tau_k^n\wedge T=T \} $ is finite a.s.~and, for $t\in\left[ 0,T\right] $,
\begin{align*} 
F_t(X^n)-F_0(X^n)=&\sum_{k=0}^{K^n-1}F_{\tau_{k+1}^n\wedge t}(X^n)-F_{\tau_{k}^n\wedge t}(X^n)\\
=&\sum_{k=0}^{K^n-1}\big[F_{\tau_{k+1}^n\wedge t}(X^n)-F_{\tau_{k+1}^n\wedge t}(X^{n-})-\Delta X_{\tau_{k+1}^n\wedge t}^n\nabla_\xr F_{\tau_{k+1}^n\wedge t}(X^{n-})\\
&~~~~~~~~+F_{\tau_{k+1}^n\wedge t}(X^{n-})-F_{\tau_{k}^n\wedge t}(X^n)+\Delta X_{\tau_{k+1}^n\wedge t}^n\nabla_\xr F_{\tau_{k+1}^n\wedge t}(X^{n-})\big]\\
=& R^{1,n}_t+R^{2,n}_t+R^{3,n}_t
\end{align*}
where 
\begin{align*} 
R^{1,n}_t=&\int_{\left] 0,t\right] \times\mathbb{R}^d}(F_s(X^{n-}\oplus_sx)-F_s(X^{n-})-x\nabla_\xr F_{s}(X^{n-}))\mathbb{1}_{\left\lbrace \lvert{x}\rvert>{1}\right\rbrace }\mu^X(ds,dx)\\
&+\int_{\left] 0,t\right] \times\mathbb{R}^d}(F_s(X^{n-}\oplus_sx)-F_s(X^{n-}))\mathbb{1}_{\left\lbrace \epsilon_n\leq\lvert{x}\rvert\leq{1}\right\rbrace }(\mu^X-\nu^X)(ds,dx)\\
&-\int_{\left] 0,t\right] \times\mathbb{R}^d}x\nabla_\xr F_{s}(X^{n-})\mathbb{1}_{\left\lbrace \epsilon_n\leq\lvert{x}\rvert\leq{1}\right\rbrace }(\mu^X-\nu^X)(ds,dx)\\
R^{2,n}_t=&\sum_{k=0}^{K^n-1}\left[ F_{\tau_{k+1}^n\wedge t}(X^{n-})-F_{\tau_{k}^n\wedge t}(X^n)+\Delta M^{n}_{\tau_{k+1}^n\wedge t}\nabla_\xr F_{\tau_{k+1}^n\wedge t}(X^{n-})\right] \\
R^{3,n}_t=&\int_{\left] 0,t\right] \times\mathbb{R}^d}(F_s(X^{n-}\oplus_sx)-F_s(X^{n-})-x\nabla_\xr F_{s}(X^{n-}))\mathbb{1}_{\left\lbrace \epsilon_n\leq\lvert{x}\rvert\leq{1}\right\rbrace }\nu^X(ds,dx)\\
&+\sum_{k=0}^{K^n-1}\Delta A^{n}_{\tau_{k+1}^n\wedge t}\nabla_\xr F_{\tau_{k+1}^n}(X^{n-}), 
\end{align*}
in which $M^{n}$ and $A^{n}$ denote respectively the martingale and the bounded variation part of $X^{n}$. 
By hypothesis, for all $k=0,\ldots,K^n-1$, the couple $(F,X^n\oplus_{\tau_k^n}(X^c-X^c_{\tau_k^n}))$ satisfies Assumption $(A)$. Moreover, by definition of $X^n$ and  $(\tau_k^n)_{k\in\N}$,   $X^n\boxplus_{\tau_{k+1}^n}X^n_{\tau_{k+1}^n-}$ is continuous on $\left[ \tau_k^n,\tau_{k+1}^n\right]$ and coincides with  $X^n\oplus_{\tau_k^n}(X^c-X^c_{\tau_k^n})$ on $\left[ 0,\tau_{k+1}^n\right]$. By Proposition \ref{ItoC01standard}, we can then find  an adapted orthogonal process $\Gamma^{F,n,k}$  such that
\begin{equation}
F_{t}(X^n\boxplus_{\tau_{k+1}^n}X^n_{\tau_{k+1}^n-})-F_{\tau_{k}^n}(X^n)=\int_{\tau_k^n}^t\nabla_\xr F_s(X^{n-})dM^{c}_s+\Gamma^{F,n,k}_t-\Gamma^{F,n,k}_{\tau_{k}^n}~~~\forall t\in\left[ \tau_k^n,\tau_{k+1}^n\right]
\end{equation}
in which we used that   $X^n$ and $X$ have the same continuous martingale part. 
By continuity of $F$ and the path  of $X^n\boxplus_{\tau_{k+1}^n}X^n_{\tau_{k+1}^n-}$ on $\left[ \tau_{k}^n,\tau_{k+1}^n\right]$, we see from the above that $\Gamma^{F,n,k}$ is continuous on $\left[\tau_{k}^n,\tau_{k+1}^n\right]$.
Then{,}  
\begin{align*}
R^{2,n}_t&=\sum_{k=0}^{K^n-1}\int_{\tau_k^n\wedge t}^{\tau^n_{k+1}\wedge t}\nabla_\xr F_s(X^{n-})dM^{c}_s+\Gamma^{F,n,k}_{\tau_{k+1}^n\wedge t}-\Gamma^{F,n,k}_{\tau_{k}^n\wedge t}+\Delta M^n_{\tau_{k+1}^n\wedge t}\nabla_\xr F_{\tau_{k+1}^n\wedge t}(X^{n-})\\
&=\int_0^t \nabla_\xr F_s(X^{n-})dM^n_s+\sum_{k=0}^{K^n-1}\Gamma^{F,n,k}_{\tau_{k+1}^n\wedge t}-\Gamma^{F,n,k}_{\tau_{k}^n\wedge t}.
\end{align*}
Let us define 
\begin{equation*}
\Gamma^{F,n}_t=R^{3,n}_t+\sum_{k=0}^{K^n-1}\Gamma^{F,n,k}_{\tau_{k+1}^n\wedge t}-\Gamma^{F,n,k}_{\tau_{k}^n\wedge t}, \;t\le T.
\end{equation*}
It follows from the above that  $\Gamma^{F,n}$ is predictable as a sum of predictable processes. 

{\rm 3.} Let us now show that 
\begin{equation}\label{conv_int_sto}
\int_0^\cdot \nabla_\xr F_s(X^{n-})dM^n_s\to\int_0^\cdot \nabla_\xr F_s(X^-)dM_s ~~\mbox{u.c.p.}
\end{equation}
on $[0,T]$.
We have
\begin{equation*}
\int_0^t \nabla_\xr F_s(X^{n-})dM^n_s=\int_0^t \nabla_\xr F_s(X^{n-})dM_s-\int_0^t \nabla_\xr F_s(X^{n-})dY^n_s
\end{equation*}
Since $Y^n$ is a purely discontinuous martingale such that $\lvert\!\lvert Y^n\rvert\!\rvert\to0$ a.s.~and $\nabla_\xr F$ is locally bounded, we can assume, up to using a localizing sequence, that   $(\nabla_\xr F(X^{n-}),  Y^n,\left[ Y^n\right] )_{n}$ is uniformly bounded by a constant $C$. Then, since $\lvert\!\lvert X^n-X \rvert\!\rvert\to 0$ a.s.~and $\nabla_\xr F$ is continuous, we deduce from \cite[Theorem~I.4.31]{jacod2013limit} that
\begin{equation*} 
\int_0^\cdot \nabla_\xr F_s(X^{n-})dM_s\to\int_0^\cdot \nabla_\xr F_s(X^-)dM_s\;\;\;\mbox{u.c.p.}
\end{equation*} 
on $[0,T]$. 
Moreover, 
\begin{align*}
\mathbb{E}\left[ \underset{t\in\left[ 0,T\right]}{\sup} \lvert\int_0^{t} \nabla_\xr F_s(X^{n-})dY^n_s\rvert^2\right] &\leq 4\mathbb{E}\left[ \int_0^{T} (\nabla_\xr F_s(X^{n-}))^2 d\left[ Y^n\right] _s\right]\\
&\leq 4C^{2}\mathbb{E}\left[ \left[ Y^n\right] _{T}\right] 
\end{align*}
in which the last term tends to $0$ as $n$ goes to $+\infty$, by dominated convergence. This proves \eqref{conv_int_sto}.

Similarly, by applying  Lemma \ref{Lconv} below, we deduce that  $R^{1,n}$ converges u.c.p.~on $[0,T]$ to 
\begin{align*}
t\in [0,T]\mapsto&  \int_{\left] 0,t\right] \times\mathbb{R}^d}(F_s(X{^{-}}\oplus_s x)-F_s(X^-)-x\nabla_\xr F_{s}(X^-))\mathbb{1}_{\left\lbrace \lvert{x}\rvert>{1}\right\rbrace }\mu^X(ds,dx)\\
&+\int_{\left] 0,t\right] \times\mathbb{R}^d}(F_s(X{^{-}}\oplus_s x)-F_s(X^-))\mathbb{1}_{\left\lbrace \lvert{x}\rvert\leq{1}\right\rbrace }(\mu^X-\nu^X)(ds,dx)\\
&-\int_{\left] 0,t\right] \times\mathbb{R}^d}x\nabla_\xr F_{s}(X^-)\mathbb{1}_{\left\lbrace \lvert{x}\rvert\leq{1}\right\rbrace }(\mu^X-\nu^X)(ds,dx). 
\end{align*}

Finally, since $\lvert\!\lvert X^n-X \rvert\!\rvert\to 0$ a.s., we have $\lvert\!\lvert F_{\cdot}(X^n)-F_{\cdot}(X)\rvert\!\rvert\to 0$ a.s.~by local uniform continuity of $F$.

\noindent{\rm 4.} Combining steps 1.~to 3.~above, we obtain that the sequence of predictable processes $(\Gamma^{F,n})_{n\in \N}$
 converges to   $ \Gamma^F$ u.c.p., which implies that $ \Gamma^F$ is predictable, and concludes the proof. 
\end{proof}

We conclude this section with the proof of the technical lemma that was used in the proof of Theorem  \ref{TIto}. We borrow the standard notations $\mathcal{A}_{loc}^+$ and $\mathcal{G}_{loc}^2(\mu^X)$ from \cite[Section I.3.a., Section II.1.d.]{jacod2013limit}.

\begin{Lemma}\label{Lconv}
Let $(\epsilon_n,X^{n},Y^{n},Z^{n})_{n\in\mathbb{N}}$ be as in the proof of Theorem \ref{TIto}.\\
Define $H^n_s(x)=(F_s(X^{n-}\oplus_sx)-F_s(X^{n-})-x\nabla_\xr F_{s}(X^{n-}))\mathbb{1}_{\left\lbrace \epsilon_n\leq\lvert{x}\rvert\leq 1 \right\rbrace }$ for $(s,x)\in [0,T]\x \R^{d}$.  Then, $H^n_s(x)\ast(\mu^X-\nu^X)$ is a sequence of purely discontinuous local martingales that converges to\\ $t\mapsto\int_{\left] 0,t\right] \times\mathbb{R}^d}(F_s(X{^{-}}\oplus_s x)-F_s(X^-)-x\nabla_\xr F_{s}(X^-))\mathbb{1}_{\left\lbrace \lvert{x}\rvert\leq{1}\right\rbrace }(\mu^X-\nu^X)(ds,dx)$ u.c.p.
\end{Lemma}
\begin{proof}
Let us define
\begin{align*}
&V^{1,n}_s(x)=(F_s(X^{-}\oplus_sx)-F_s(X^{n-}\oplus_sx)+F_s(X^{n-})-F_s(X^{-})+x\nabla_\xr F_{s}(X^{n-})-x\nabla_\xr F_{s}(X^{-}))\mathbb{1}_{\left\lbrace \epsilon_n\leq\lvert{x}\rvert\leq 1 \right\rbrace }\\
&V^{2,n}_s(x)=(F_s(X^-\oplus_sx)-F_s(X^-)-x\nabla_\xr F_{s}(X^-))\mathbb{1}_{\left\lbrace \lvert{x}\rvert<\epsilon_n\right\rbrace } 
\end{align*}
for $(s,x)\in [0,T]\x \R^{d}$. By linearity, it suffices to show that ${I^{i,n}}:={V^{i,n}}\ast(\mu^X-\nu^X)$  converge to $0$ u.c.p., for $i=1,2$.\\
We recall that any càglàd process is locally bounded.  Furthermore,  since $X$ is càdlàg and has finite quadratic variation, we have $\sum_{s\in\left[ 0,T\right]}\lvert\Delta X_s\rvert^2<+\infty$ a.s.~by \cite[Lemma~2.10]{bandini2017weak}. We also recall that $(Z^n)_{n\in\N}$ is uniformly locally bounded{, see Remark \ref{rem : def Zn}}.  Consider the c\`adl\`ag process $E:=(X_{t-},\sum_{s<t} \lvert\Delta X_s\rvert ^2)_{t\geq 0}$ and let $(S_m)_{m\in\N}$ be a localization sequence  such that for all $m\in\N$ the processes $((Z^n,E)_{\cdot\wedge S^m}\mathbb{1}_{S^m>0})_{n\in\N}$ are uniformly bounded in $n$. It suffices to show that ${(V^{i,n}\ast(\mu^X-\nu^X))_{\cdot\wedge S}}$  converge to $0$ u.c.p. for a fixed  $S=S^m$, $i=1,2$. Let $C$ be such that $\|E_{\cdot \wedge S}\|\vee\|Z^n_{\cdot \wedge S}\|\leq{C}$ for all $n$ a.s.
Then, 
\begin{align} 
\mathbb{E}\left[ ({\lvert V^{2,n}\rvert ^2\ast \nu^X})_{T\wedge S}\right]&=\mathbb{E}\left[\int_{\left] 0,T\wedge S\right] \times\mathbb{R}^d}\lvert (F_s(X^-\oplus_sx)-F_s(X^-)-x\nabla_\xr F_{s}(X^-))\mathbb{1}_{\left\lbrace \lvert{x}\rvert<\epsilon_n\right\rbrace }\rvert ^2\mu^X(ds,dx)\right]\nonumber\\
&=\mathbb{E}\left[\sum_{\substack{s\leq T\wedge S\\0<\lvert \Delta{X_s}\rvert<{\eps_n}}}\lvert (F_s(X)-F_s(X^-)-\Delta X_s\nabla_\xr F_{s}(X^-))\rvert ^2\right]\nonumber\\
&=\mathbb{E}\left[\sum_{\substack{s\leq T\wedge S\\0<\lvert \Delta{X_s}\rvert<{\eps_n}}} \lvert\Delta{X_s}\rvert^2\;\lvert\int_0^1\{\nabla_\xr{F}_s(X^-\oplus_{s}\lambda\Delta{X_s})-\nabla_\xr F(X^-)\}d\lambda\rvert^2\right]\nonumber\\
&\leq\delta^2_{C +1}(\nabla_\xr F,\epsilon_n)\mathbb{E}\left[\sum_{\substack{s\in\left]0,T\wedge S\right[\\0<\lvert \Delta{X_{s}}\rvert<{\eps_n}}}  \lvert\Delta{X_{s}}\rvert^2\mathbb{1}_{S>0}+\lvert\Delta{X_{T\wedge S}}\rvert^2\mathbb{1}_{S>0}\mathbb{1}_{\lvert\Delta{X}_{T\wedge S}\rvert\leq{1}}\right]\nonumber\\
&\leq\delta^2_{C +1}(\nabla_\xr F,\epsilon_n) (C+1)\label{eq: inega borne avec C 1}
\end{align} 
where $\delta_\cdot(\nabla_\xr F,\cdot)$ denotes the modulus of continuity of $\nabla_\xr F$ defined in \eqref{def_module_continuit}.
In the same way,
\begin{align} \label{eq: inega borne avec C 2}
\mathbb{E}&\left[ ({\lvert V^{1,n}\rvert ^2\ast \nu^X})_{T\wedge S}\right]\nonumber\\&=\mathbb{E}\left[\sum_{\substack{{s\leq T\wedge S}\\{0<\lvert \Delta{X_s}\rvert\le 1} }} \lvert\Delta{X_s}\rvert^2\;\lvert\int_0^1\{\nabla_\xr{F}_s(X^-\oplus_{s}\lambda\Delta{X_s})-\nabla_\xr F(X^{n-}\oplus_{s}\lambda\Delta{X_s})\}d\lambda+\nabla_\xr F(X^{n-})-\nabla_\xr F(X^{-})\rvert^2\right]\nonumber\\
&\leq 4(C+1)\mathbb{E}[\delta^2_{C +1}(\nabla_\xr F,\|Z^n\|_{[0,S\wedge T]})]
\end{align}
where $\|\xr\|_{[0,t]}=\sup_{s\in[0,t]}\lvert\xr_s\rvert$ for all $(t,\xr)\in\Theta$.\\
Thus,  for $i=1,2$, $V^{i,n}$   belongs to $\mathcal{G}^2_{loc}(\mu^X)$ by \cite[Lemma~2.4]{bandini2018special}, and $I^{i,n}_{\cdot\wedge S}$ is a purely discontinuous square integrable martingale by \cite[Theorem~11.21]{he2019semimartingale}.  Also,  by \cite[3) of Theorem 11.21]{he2019semimartingale},we have \begin{equation}\label{eq: borne [In]}
 \left[ I^{i,n}\right]_{t\wedge S}  =\ \int_{\left] 0,t\wedge S\right] \times\mathbb{R}^d}\lvert{V^{i,n}_s(x)}\rvert^{2}   \nu^{X}(ds,dx)-\sum_{0<s\leq t\wedge S}\lvert\hat V^{i,n}_s\rvert^2 \leq \int_{\left] 0,t\wedge S\right] \times\mathbb{R}^d}\lvert{V^{i,n}_s(x)}\rvert^{2}   \nu^{X}(ds,dx)
\end{equation}
where $\hat V^{i,n}_s=\int_{\mathbb{R}^d}V^{i,n}_s(x)\nu(\left\lbrace s\right\rbrace ,dx)$, for $i=1,2$.  \\
Hence, we can apply  Doob's maximal inequality to the square integrable martingale $I^{i,n}_{\cdot\wedge S}${, $i=1,2$,} and then use \eqref{eq: inega borne avec C 1}, \eqref{eq: inega borne avec C 2} and \eqref{eq: borne [In]} to obtain that, for any $\alpha>0$, 
\begin{align*}
&\mathbb{P}(\sup_{t\in\left[ 0,T\right]}\lvert{I^{2,n}_{t\wedge S}}\rvert\geq \alpha)\leq\dfrac{1}{\alpha^2}\mathbb{E}\left[ \lvert  I^{2,n}_{T\wedge S}\rvert ^2\right] \leq \dfrac{1}{\alpha^2}\mathbb{E}\left[\int_{\left] 0,T\wedge S\right] \times\mathbb{R}^d}\lvert{V^{2,n}_s(x)}\rvert^{2}   \nu^{X}(ds,dx)\right] \leq\dfrac{1}{\alpha^2} \delta^2_{C+1}(\nabla_\xr F,\epsilon_n) (C+1){,}\\
&\mathbb{P}(\sup_{t\in\left[ 0,T\right]}\lvert{I^{1,n}_{t\wedge S}}\rvert\geq \alpha) \leq \dfrac{1}{\alpha^2}\mathbb{E}\left[\int_{\left] 0,T\wedge S\right] \times\mathbb{R}^d}\lvert{V^{1,n}_s(x)}\rvert^{2}   \nu^{X}(ds,dx)\right] \leq\dfrac{4}{\alpha^2}\mathbb{E}[\delta^2_{C +1}(\nabla_\xr F,\|Z^n\|_{[0,S\wedge T]})](C+1).
\end{align*}
The right-hand side terms tend to $0$ as $n\to \infty$ (by {Remark  \ref{rem : def Zn} and} dominated convergence for the second one), which concludes the proof.
\end{proof}

\section{A toy example of application}

To illustrate our main result, we now provide a simple toy example of application. We keep it as simple as possible. Semilinear and fully-non linear problems have been studied in  \cite{bouchard2021ac,bouchard2021approximate} in the context of continuous path processes and can also be extended to our setting. 

We fix $d=1$. Let $W$ be a standard Brownian motion and $N$ be a compound Poisson process with compensator $\nu_{t}dt$, for some predictable $(t,\omega)\in [0,T]\x \Omega \mapsto \nu_{t}(\omega,\cdot)$ taking values in the set of probability measures on $\R$. Given $(t,\xr)\in \Theta$, we define 
$X^{t,\xr}$ by 
\begin{equation}
X^{t,\xr}_s=\xr_{t\wedge s}+A_{t\vee s}-A_{t}+\int_{t}^{t\vee s}\sigma_sdW_s+\int_{{\left] s,t\vee s\right]} \times\R^d}\gamma_s(y)N(ds,dy),\;\;s\le T,
\end{equation}
 where   $\sigma$ is predictable and bounded, $\gamma$ is $\Pc\otimes \Bc(\R)$-mesurable\footnote{We use the standard notations $\Pc$ (resp.~$\Bc(R)$) for the predictable sigma-field (resp.~the Borel sigma-field).} and bounded, and $A$ is a c\`adl\`ag, bounded, and with bounded variations predictable process.

We then consider a bounded $C^{1+\alpha}(\R)$-map $g:\R\to \R$, for some $\alpha\in (0,1]$, with bounded derivative, and a right-continuous    {measure} $\mu$ with   bounded total variation on $\left[ 0,T\right]$ and at most a finite number of atoms $\{0\le t_{1}\le \cdots \le t_{n}\le T\}$ which are deterministic. We define
\begin{align*}
\vr:~&(t,\xr)\in \Theta\mapsto\E[g(\int_0^TX^{t,\xr}_s\mu(ds))].
\end{align*}
 The following is nothing but a version of the celebrated Clark's formula, see \cite{clark1970representation}, which we retrieve here as a consequence of Theorem \ref{TIto}.
 
 \begin{Proposition} Let the above conditions hold and set $X:=X^{0,\xr}$ for some $\xr \in D([0,T])$. Then, $\vr$ admits a vertical derivative
  $$
\nabla_{\xr}\vr: (t',\xr')\in \Theta\mapsto \E[\nabla g(\int_0^TX^{t',\xr'}_s\mu(ds))\mu([t',T])], 
 $$
 and there exists an orthogonal and predictable process $\Gamma$ such that $\Gamma_{0}=0$ and 
 \begin{align*} 
 g(\int_0^TX_s\mu(ds))=&\vr(0,\xr)+\int_{0}^{T} \nabla_{\xr}\vr(s, X)\sigma_{s}dW_{s}\\
 &+\sum_{s\le T} (\vr(s,X)-\vr(s,X^{-}))-\int_{0}^{T}\int_{\R} (\vr(s,X^{-}\oplus_{s} y)-\vr(s,X^{-}))\nu_{s}(dy) \lambda_{s} ds+\Gamma_{T}.
 \end{align*}
 If moreover $\vr(\cdot,X)$ is a martingale, then $\Gamma\equiv 0$. It is in particular the case if $A$, $\sigma$, $\gamma$, $\lambda$ and $t\in [0,T]\mapsto \nu_{t}$ are deterministic.   
 \end{Proposition}
 
\begin{proof} 1. We first assume that $\mu$ does no have atoms. How to treat the general case will be discussed in step 3.  
First note that, for $(t,\xr)\in \Theta$ and $y\in \R$, 
$$
\left|\vr(t,\xr\oplus_{t}y)-\vr(t,\xr)-\E[\nabla g(\int_0^TX^{t,\xr}_s\mu(ds))\mu([t,T])]y\right|
\le C\E[\{|\mu|([t,T]) |y|\}^{1+\alpha}]
$$ 
for some $C>0$. Since $|\mu|$ is bounded, this implies that 
$$
\nabla_{\xr}\vr(t,\xr)=\E[\nabla g(\int_0^TX^{t,\xr}_s\mu(ds))\mu([t,T])]. 
$$
Clearly, $\vr$ and $\nabla_{\xr}\vr$ are  locally uniformly bounded since $g$, $\nabla g$ and $|\mu|$ are bounded. 

2. Note that  $(\vr,Z)$ satisfies assumption $(A)$ for all càdlàg process $Z$ by Remark \ref{suff_Ito}.
We now prove that 
$\vr$  and $\nabla_{\xr}\vr$  are locally uniformly continuous.  Fix $(t,\xr),(t',\xr')\in\Theta$ with $t'\ge t$. Then, 
by standard estimates based on our boundedness assumptions,
\begin{align*}
\E[\int_0^T\lvert X^{t,\xr}_s-X^{t',\xr'}_s\rvert\mu(ds)]\leq&  C(\lvert\!\lvert\xr_{t\wedge}-\xr '_{t'\wedge}\rvert\!\rvert+\sqrt{t'-t})
\end{align*}
for some $C>0$ that does not depend on $(t,\xr)$ and $(t',\xr')$. Given the above and the fact that $g$ is Lipschitz and $C^{1+\alpha}(\R)$, this implies that 
\begin{align}\label{eq: estim regul v} 
|\vr(t',\xr')-\vr(t,\xr)|&\le C(\lvert\!\lvert\xr_{t\wedge}-\xr '_{t'\wedge}\rvert\!\rvert +(t'-t)^{\frac{1}2})\\
|\nabla_{\xr}\vr(t',\xr')-\nabla_{\xr}\vr(t,\xr)|&\le C(\lvert\!\lvert\xr_{t\wedge}-\xr '_{t'\wedge}\rvert\!\rvert^{\alpha}+(t'-t)^{\frac{\alpha}2}+\mu([t,t']))\label{eq: estim regul nabla v} 
\end{align}
for some $C>0$ that does not depend on $(t,\xr)$ and $(t',\xr')$.

3.  In the general case where $\mu$ has a finite number of atoms $\{0\le t_{1}\le \cdots \le t_{n}\le T\}$, then \eqref{eq: estim regul v}-\eqref{eq: estim regul nabla v}  shows that $\vr$ and $\nabla_{\xr}\vr$ are  locally uniformly bounded and  locally uniformly continuous on each closed and convex interval of $\cup_{i=0}^{n} [t_{i},t_{i+1})$, with the convention that $t_{0}=0$ and $t_{n+1}=T$. Moreover, \eqref{eq: estim regul v} also shows that $\vr(\cdot,X^{-})$ admit right-limits, for all $X:=X^{0,\xr}$ for some $\xr \in D([0,T])$. Then, one can apply Theorem \ref{TIto} on intervals of the form $[t_{i},t]$ with $t_{i}\le t<t_{i+1}$ and $0\le i\le n$. Since, like $X$, $\vr(\cdot,X)$ is a.s.~continuous at $t_{i+1}$, this implies that 
\begin{align*} 
\vr(t\wedge t_{i+1}, X)=&\vr(t_{i}, X)+\int_{t_{i}}^{t\wedge t_{i+1}} \nabla_{\xr}\vr(s, X)dM^{c}_{s}\\
&+\sum_{s\le t\wedge t_{i+1}} (\vr(s,X)-\vr(s,X^{-}))-\int_{0}^{t\wedge t_{i+1}}\int_{\R} (\vr(s,X^{-}\oplus_{s} y)-\vr(s,X^{-}))\nu_{s}(dy) \lambda_{s} ds\\
&+\Gamma_{t\wedge t_{i+1}}-\Gamma_{t_{i}}, \; t\in [t_{i},t_{i+1}],
\end{align*}
in which $M^{c}$ is the continuous martingale part of $X$ and $\Gamma$ is predictable and orthogonal.

4. In the case where $A$, $\sigma$, $\gamma$, $\lambda$ and $t\in [0,T]\mapsto \nu_{t}$  are deterministic, then one easily checks that $\vr(\cdot,X)$ is a martingale. If the later holds, then $\Gamma\equiv 0$ by uniqueness of the martingale decomposition.
\end{proof}

\appendix
\section{Appendix}

We first state a technical result whose proof is very close to the first part of the proof of Lemma \ref{Lconv}. Again, we borrow the standard notations $\mathcal{A}_{loc}^+$ and $\mathcal{G}_{loc}^2(\mu^X)$ from \cite[Section I.3.a., Section II.1.d.]{jacod2013limit}.

\begin{Lemma}\label{purely _discontinuous}
Let $F\in\C^{0,1}(\Theta)$ be such that $\nabla_\xr F$ is locally bounded and let $X$ be a càdlàg process such that $\sum_{s\leq T} \lvert\Delta X_s\rvert^2<+\infty$ a.s. Then, 
\begin{align*}
&V:=\lvert(F_s(X\oplus_s x)-F_s(X^-))\mathbb{1}_{\left\lbrace \lvert{x}\rvert\leq{1}\right\rbrace }\rvert^{2} \ast \mu^{X} \in \mathcal{A}_{loc}^+\\
&W:=\lvert{x}\nabla_\xr F_{s}(X^-)\mathbb{1}_{\left\lbrace \lvert{x}\rvert\leq{1}\right\rbrace }\rvert^{2} \ast \mu^{X} \in \mathcal{A}_{loc}^+
\end{align*}
In particular,  $((F_s(X\oplus_s x)-F_s(X^-))\mathbb{1}_{\left\lbrace \lvert{x}\rvert\leq{1}\right\rbrace }) \ast (\mu^{X}-\nu^X)$  and $(x\nabla_\xr F_{s}(X^-)\mathbb{1}_{\left\lbrace \lvert{x}\rvert\leq{1}\right\rbrace }) \ast (\mu^{X}-\nu^X)$ are well-defined purely discontinuous local martingales. 
\end{Lemma}

\begin{proof}
The fact that both processes are increasing is trivial. We next argue as in the proof of Lemma \ref{Lconv}.
Set $Y:=(X_{t-},\sum_{s<t} \lvert\Delta X_s\rvert ^2)_{t\geq 0}$ and let $(S_m)_{m\in\N}$ be a localization sequence such that $(Y_{\cdot\wedge S^m}\mathbb{1}_{S^m>0})_{m\in\N}$ is a sequence of bounded processes. We fix $S=S_m$ for some $m$ and $C$ s.t. $\lvert{Y_{t\wedge S}}\rvert\leq{C}~~\forall t\leq T $ a.s. 

Then, 
\begin{align*}
\mathbb{E}\left[   W_{t\wedge{S}}\right] &=\mathbb{E}\left[   \int_{\left] 0,t\wedge S\right] \times\mathbb{R}}\lvert{x}\nabla_\xr F_{s}(X^-)\mathbb{1}_{\left\lbrace \lvert{x}\rvert\leq{1}\right\rbrace }\rvert^{2}   \mu^{X}(ds,dx)\right]\\
&\leq\sup_{s\in\left[ 0,T\right],~\lvert\!\lvert\xr\rvert\!\rvert\leq{C}}|\nabla_\xr F_s(\xr)|^2\;\mathbb{E}\left[\sum_{\substack{s\in\left] 0,t\wedge S\right[\\0<\lvert \Delta{X_{s}}\rvert\leq{1}}} \lvert\Delta{X_{s}}\rvert^2\mathbb{1}_{S>0}+\lvert\Delta{X_{S}}\rvert^2\mathbb{1}_{S>0}\mathbb{1}_{\lvert\Delta{X}_S\rvert\leq{1}}\right]\\
&\leq\sup_{s\in\left[ 0,T\right],~\lvert\!\lvert\xr\rvert\!\rvert\leq{C}}|\nabla_\xr F_s(\xr)|^2\;(C+1)
. 
\end{align*} 
The last term is finite since $\nabla_\xr F$ is locally  bounded.
Similarly, 
\begin{align*}
\mathbb{E}\left[   V_{t\wedge{S}}\right] &=\mathbb{E}\left[   \int_{\left] 0,t\wedge S\right] \times\mathbb{R}^d}\lvert(F_s(X\oplus_s x)-F_s(X^-))\mathbb{1}_{\left\lbrace \lvert{x}\rvert\leq{1}\right\rbrace }\rvert^{2}  \mu^{X}(ds,dx)\right]\\
&=\mathbb{E}\left[\sum_{\substack{s\in\left]0,t\wedge S\right]\\0<\lvert \Delta{X_{s}}\rvert\leq{1}}} \lvert(F_{s}(X)-F_{s}(X^-)\rvert^2\right]\\
&=\mathbb{E}\left[\sum_{\substack{s\in\left]0,t\wedge S\right]\\0<\lvert \Delta{X_{s}}\rvert\leq{1}}}  \lvert\Delta{X_{s}}\rvert^2\;\lvert\int_0^1\nabla_x{F}_{s}(X^-\oplus_{s}\lambda\Delta X_s)d\lambda\rvert^2\right]\\
&\leq\sup_{s\in\left[ 0,T\right],~\lvert\!\lvert{\xr}\rvert\!\rvert\leq{M}}{|\nabla_\xr F_s(\xr)|}^2 \;\mathbb{E}\left[\sum_{\substack{s\in\left]0,t\wedge S\right[\\0<\lvert \Delta{X_{s}}\rvert\leq{1}}}  \lvert\Delta{X_{s}}\rvert^2\mathbb{1}_{S>0}+\lvert\Delta{X_{S}}\rvert^2\mathbb{1}_{S>0}\mathbb{1}_{\lvert\Delta{X}_S\rvert\leq{1}}\right]\\
&\leq\sup_{s\in\left[ 0,T\right],~\lvert\!\lvert{\xr}\rvert\!\rvert\leq{C}}|\nabla_\xr F_s(\xr)|^2\; (C+1) .
\end{align*} 

Thus, $V$ and $W$ belong to $\mathcal{A}^+_{loc}$.

We conclude that $((F_s(X\oplus_s x)-F_s(X^-))\mathbb{1}_{\left\lbrace \lvert{x}\rvert\leq{1}\right\rbrace }) \ast (\mu^{X}-\nu^X)$  and $(x\nabla_\xr F_{s}(X^-)\mathbb{1}_{\left\lbrace \lvert{x}\rvert\leq{1}\right\rbrace }) \ast (\mu^{X}-\nu^X)$ are well-defined square integrable purely discontinuous locale martingales by \cite[Theorem~11.21]{he2019semimartingale} since their integrands belong to $\mathcal{G}^2_{loc}(\mu^X)$ by  \cite[Lemma~2.4]{bandini2018special}.\\
\end{proof}

The next result follows from the same arguments as in   \cite{bandini2017weak,bouchard2021ac}. {At the difference of Theorem \ref{TIto}, it does not assert that $\Gamma^{F}$ is predictable.} We provide its proof for completeness. 

\begin{Proposition}\label{ItoC01standard}
Let $X=X_0+M+A$ be a càdlàg weak Dirichlet process with finite quadratic variation. Let $\mu^X$ be its jump measure and $\nu^X$ its compensator. 
 
 Let $F:\Theta\to\mathbb{R}$ be $\mathbb{C}^{0,1}$, such that $F$ and $\nabla_{x}F$ are both in $\mathbb{C}_{loc}^{u,b}(\Theta)$, and such that $s\mapsto\nabla_\xr F_{s}(X^-)$ admits right-limits a.s. Then, $(F_{t}(X))_{t\in{\left[ 0,T \right] }}$ is a weak Dirichlet process with decomposition
\begin{align*}
F_{t}(X)=F_{0}(X)+&\int_0^t\nabla_\xr F_{s}(X^-) dM_s\notag\\
&+\int_{\left] 0,t\right] \times\mathbb{R}^d}(F_s(X^-\oplus_s x)-F_s(X^-)-x\nabla_\xr F_{s}(X^-))\mathbb{1}_{\left\lbrace \lvert{x}\rvert>{1}\right\rbrace }\mu^X(ds,dx)\notag\\
&+\int_{\left] 0,t\right] \times\mathbb{R}^d}(F_s(X^-\oplus_s x)-F_s(X^-))\mathbb{1}_{\left\lbrace \lvert{x}\rvert\leq{1}\right\rbrace }(\mu^X-\nu^X)(ds,dx)\notag\\
&-\int_{\left] 0,t\right] \times\mathbb{R}^d}x\nabla_\xr F_{s}(X^-)\mathbb{1}_{\left\lbrace \lvert{x}\rvert\leq{1}\right\rbrace }(\mu^X-\nu^X)(ds,dx)\notag\\
&+\Gamma_t^F,~~\forall t\in\left[ 0,T\right],
\end{align*}
where $\Gamma^F$ is an orthogonal process, if and only if $(F,X)$ satisfies Assumption $(A)$.
\end{Proposition}

\begin{proof} 
For the rest of the proof, we denote by  $\delta_\cdot(F,\cdot)$ and $\delta_\cdot(\nabla_{\xr}F,\cdot)$ the modulus of continuity of $F$ and $\nabla_{\xr}F$, see  \eqref{def_module_continuit}.

Let $N$ be a continuous local martingale. Our aim is to show that $\left[ \Gamma^F,N\right]\equiv0$, in which 
\begin{align}\label{Gamma}
\Gamma_t^F=F_{t}(X)-F_{0}(X)&-\int_0^t\nabla_\xr F_{s}(X^-)dM_s\notag\\
&-\int_{\left] 0,t\right] \times\mathbb{R}}(F_s(X\oplus_s x)-F_s(X^-)-x\nabla_\xr F_{s}(X^-))\mathbb{1}_{\left\lbrace \lvert{x}\rvert>{1}\right\rbrace }\mu^X(ds,dx)\notag\\
&-\int_{\left] 0,t\right] \times\mathbb{R}}(F_s(X\oplus_s x)-F_s(X^-))\mathbb{1}_{\left\lbrace \lvert{x}\rvert\leq{1}\right\rbrace }(\mu^X-\nu^X)(ds,dx)\notag\\
&+\int_{\left] 0,t\right] \times\mathbb{R}}x\nabla_\xr F_{s}(X^-)\mathbb{1}_{\left\lbrace \lvert{x}\rvert\leq{1}\right\rbrace }(\mu^X-\nu^X)(ds,dx),~~t\leq T.
\end{align}
Note that $\sum_{s\leq T} \lvert\Delta X_s\rvert^2<+\infty$ a.s., since $X$ has finite quadratic variation,  see \cite[Lemma~2.10.]{bandini2017weak}.
Then, by Lemma \ref{purely _discontinuous}  and the definition of a purely discontinuous local martingale, the two last terms of \eqref{Gamma} are orthogonal{,} hence their quadratic {co}variation with $N$ equals $0$. 

On the orher hand,  since $X$ is a càdlàg process, it has finitely many jumps larger or equal to $1$, a.s. \\ Hence $\int_{\left] 0,\cdot\right] \times\mathbb{R^d}}(F_s(X\oplus_s x)-F_s(X^-)-x\nabla_\xr F_{s}(X^-))\mathbb{1}_{\left\lbrace \lvert{x}\rvert>{1}\right\rbrace }\mu^X(ds,dx)$ is a bounded variation process and, by Remark  \ref{consistence_bracket}, its quadratic covariation with $N$ also equals $0$. Moreover, by Remark \ref{consistence_bracket},
\begin{equation*}
\left[ \int_0^{\cdot}\nabla_\xr F_{s}(X^-) dM_s,N\right]  _t=\int_0^t\nabla_\xr F_{s}(X^-) d\left[ M,N\right]  _s.
\end{equation*}
Thus,  by bilinearity of the quadratic covariation, we only have to show that
\begin{equation*}
\left[ F_\cdot(X),N\right]_t=\int_0^t\nabla_\xr F_{s}(X^-) d\left[ M,N\right]  _s
\end{equation*}
which by continuity of $N$ and    \cite[Proposition A.3]{bandini2017weak} is equivalent to 
\begin{equation*}
I^{\epsilon}_t:=\dfrac{1}{\epsilon}\int_0^t(F_{s+\epsilon}(X)-F_s(X))(N_{s+\epsilon}-N_s)ds\underset{\epsilon\rightarrow{0}}{\longrightarrow}
\int_0^t\nabla_\xr F_{s}(X^-) d\left[ M,N\right]  _s ~~\mbox{u.c.p.}
\end{equation*}
We have 
\begin{equation*}
I^{\epsilon}_t=I^{\epsilon,1}_t+I^{\epsilon,2}_t
\end{equation*}
where
\begin{align*}
&I^{\epsilon,1}_t=\dfrac{1}{\epsilon}\int_0^t(F_{s+\epsilon}(X_{s\wedge}\boxplus_{s+\epsilon}X_{s+\epsilon})-F_s(X))(N_{s+\epsilon}-N_s)ds\\
&I^{\epsilon,2}_t=\dfrac{1}{\epsilon}\int_0^t(F_{s+\epsilon}(X)-F_{s+\epsilon}(X_{s\wedge}\boxplus_{s+\epsilon}X_{s+\epsilon}))(N_{s+\epsilon}-N_s)ds{.}
\end{align*}
If we show that $I^{\epsilon,1}\underset{\epsilon\rightarrow{0}}{\longrightarrow}\int_0^\cdot\nabla_\xr F_{s}(X^-) d\left[ M,N\right]  _s$ u.c.p., then $I^{\epsilon}\underset{\epsilon\rightarrow{0}}{\longrightarrow}0$ u.c.p.~if and only if $I^{\epsilon,2}\underset{\epsilon\rightarrow{0}}{\longrightarrow}0$ u.c.p., which would provide the required result.

Let us decompose $I^{\epsilon,1}$ in
\begin{equation*}
I^{\epsilon,1}_t=I^{\epsilon,11}_t+I^{\epsilon,12}_t+I^{\epsilon,13}_t+I^{\epsilon,14}_t
\end{equation*}
where 
\begin{align*}
&I^{\epsilon,11}_t=\dfrac{1}{\epsilon}\int_0^t\int_0^1(\nabla_\xr F_{s+\epsilon}(X_{s\wedge}\oplus_{s+\epsilon}\lambda(X_{s+\epsilon}-X_s))-\nabla_\xr F_{s+\epsilon}(X_{s\wedge}))d\lambda~(X_{s+\epsilon}-X_s)(N_{s+\epsilon}-N_s)ds\\
&I^{\epsilon,12}_t=\dfrac{1}{\epsilon}\int_0^t(\nabla_\xr F_{s+\epsilon}(X_{s\wedge})-\nabla_\xr F_{s}(X))(X_{s+\epsilon}-X_s)(N_{s+\epsilon}-N_s)ds\\
&I^{\epsilon,13}_t=\dfrac{1}{\epsilon}\int_0^t\nabla_\xr F_{s}(X)(X_{s+\epsilon}-X_s)(N_{s+\epsilon}-N_s)ds\\
&I^{\epsilon,14}_t=\dfrac{1}{\epsilon}\int_0^t(F_{s+\epsilon}(X_{s\wedge})-F_s(X))(N_{s+\epsilon}-N_s)ds.
\end{align*} 
Since $\nabla_\xr F$ is in $\Cb_{loc}^{u,b}(\Theta)$, we have
\begin{equation*}
\lvert{I^{\epsilon,12}_t}\rvert\leq\delta_{\lvert\!\lvert{X}\rvert\!\rvert}(\nabla_\xr F,\epsilon)\sqrt{\left[ N,N\right]^{ucp}_{\epsilon,t}\;\left[ X,X\right]^{ucp}_{\epsilon,t}}.
\end{equation*}
Since $N$ is a local martingale, $N$ has finite quadratic variation by Remark \ref{consistence_bracket}. Then, $\left[ N,N\right]^{ucp}_{\epsilon}\underset{\epsilon\rightarrow{0}}{\longrightarrow}\left[ N,N\right]$ and $\left[ X,X\right]^{ucp}_{\epsilon}\underset{\epsilon\rightarrow{0}}{\longrightarrow}\left[ X,X\right]$ u.c.p. Hence, the right-hand side term converges to $0$ u.c.p.~and thus $I^{\epsilon,12}\underset{\epsilon\rightarrow{0}}{\longrightarrow}0$  u.c.p.

Let us now consider $I^{\epsilon,14}$:
\begin{align*}
I^{\epsilon,14}_t=&\dfrac{1}{\epsilon}\int_0^t(F_{s+\epsilon}(X_{s\wedge})-F_s(X))(N_{s+\epsilon}-N_s)ds\\
&=\dfrac{1}{\epsilon}\int_0^t(F_{s+\epsilon}(X_{s\wedge})-F_s(X))\int_s^{s+\epsilon}dN_u~ds\\
&=\dfrac{1}{\epsilon}\int_0^{t+\epsilon}\int_{(u-\epsilon)\vee0}^{u}F_{s+\epsilon}(X_{s\wedge})-F_s(X)ds~dN_u 
\end{align*}
where we used the stochastic Fubini's Lemma to deduce the last equation. 
Since  $F\in{\Cb}_{loc}^{u,b}(\Theta)$,  we have $\frac{1}{\epsilon}\lvert\int_{(u-\epsilon)\vee0}^{u}F_{s+\epsilon}(X_{s\wedge})-F_s(X)ds\rvert\leq\delta_{\lvert\!\lvert{X}\rvert\!\rvert}(F,\epsilon)\underset{\epsilon\rightarrow{0}}{\longrightarrow}0~~\forall{u}\in\left[ 0,T\right]$ a.s. 
Then, using \cite[Theorem~I.4.31]{jacod2013limit},  we conclude that $I^{\epsilon,14}\underset{\epsilon\rightarrow{0}}{\longrightarrow}0$ u.c.p.
By using \cite[Proposition~A.6.]{bandini2017weak},  we have  $I^{\epsilon,13}\underset{\epsilon\rightarrow{0}}{\longrightarrow}\int_0^\cdot\nabla_\xr F_{s}(X^-) d\left[ M,N\right]  _s$ u.c.p. Hence, it remains to show that   $I^{\epsilon,11}\underset{\epsilon\rightarrow{0}}{\longrightarrow}0$ u.c.p.

Let $(\epsilon_n)_{n\in\N}$ a sequence of real numbers which tends to $0$ and let $\mathcal{N}$ be an element of $\mathcal{F}$ $s.t.$ ${\mathbb P}(\mathcal{N}^c)=0$ and $s.t.$  $\left[ N,N\right]^{ucp}_{\epsilon_n}\underset{n\rightarrow{+\infty}}{\longrightarrow}\left[ N,N\right]$ and $\left[ X,X\right]^{ucp}_{\epsilon_n}\underset{n\rightarrow{+\infty}}{\longrightarrow}\left[ X,X\right]$ uniformly on $\mathcal{N}.$  We fix $\omega\in\mathcal{N}$ for the rest of the proof (we omit it to alleviate notations).\\
Fix an arbitrary $\gamma>{0}$ and let $(t_i)_{i\in\N}$ be the jump times of $X$ (depending on this fixed $\omega$). \\
By \cite[Lemma~2.10.]{bandini2017weak}, there exists $K=K(\omega)$ $s.t.$ $\sum_{i=K+1}^\infty\lvert\Delta{X_{t_i}}\rvert^2\leq\gamma^2$.\\
We define $A_{\epsilon_n}=\bigcup\limits_{i=1}^{K}\left] t_i-\epsilon_n,t_i\right] $ and $B_{\epsilon_n}=\left[ 0,T\right] \backslash{A_{\epsilon_n}}$ and decompose $I^{\epsilon_n,11}$ as follows:
\begin{equation*}
I^{\epsilon_n,11}=I^{\epsilon_n,11A}+I^{\epsilon_n,11B}
\end{equation*}
where
\begin{align*}
I^{\epsilon_n,11A}_t&=\sum_{i=1}^K\dfrac{1}{\epsilon_n}\int_{t_i-\epsilon_n}^{t_i}\mathbb{1}_{s\in\left] 0,t\right]}\int_0^1 G_{\epsilon_{n}}(s,\lambda)d\lambda~(X_{s+\epsilon_n}-X_s)(N_{s+\epsilon_n}-N_s)ds\\
I^{\epsilon_n,11B}_t&=\dfrac{1}{\epsilon_n}\int_0^t\mathbb{1}_{s\in{B}_{\epsilon_n}}\int_0^1G_{\epsilon_{n}}(s,\lambda)d\lambda~(X_{s+\epsilon_n}-X_s)(N_{s+\epsilon_n}-N_s)ds
\end{align*}
in which 
$$
G_{\epsilon_{n}}(s,\lambda):=\nabla_\xr F_{s+\epsilon_n}(X_{s\wedge}\oplus_{s+\epsilon_n}\lambda(X_{s+\epsilon_n}-X_s))-\nabla_\xr F_{s+\epsilon_n}(X_{s\wedge}).
$$
We have
\begin{align*}
\lvert{I^{\epsilon_n,11B}_t}\rvert&\leq\delta_{\lvert\!\lvert{X}\rvert\!\rvert}(\nabla_\xr F,\sup_{i~{ \rm s.t. }~ t_{i}\le T}~\sup_{r,a\in\left[ t_i,t_{i+1}\right[,~\lvert{r-a}\rvert\leq\epsilon_n}\lvert{X_r-X_a}\rvert)\sqrt{\left[ N,N\right]^{ucp}_{\epsilon_n,t}\;\left[ X,X\right]^{ucp}_{\epsilon_n,t}}\\
&\leq\delta_{\lvert\!\lvert{X}\rvert\!\rvert}(\nabla_\xr F,3\gamma)\sqrt{\left[ N,N\right]^{ucp}_{\epsilon_n,t}\;\left[ X,X\right]^{ucp}_{\epsilon_n,t}}
\end{align*}
for $n$ large enough (depending on $\omega$), by \cite[Lemma~2.12.]{bandini2017weak} applied successively on the intevals $\left[ t_i,t_{i+1}\right]$ to the processes $X_{t_i}\boxplus_{t_i+1}X_{t_{i+1}-}$ for $i=0,\ldots, K-1$ and on $\left[0,t_0\right]$ and $\left[t_K,T\right]$. Then,
\begin{equation*}
\limsup_{n\rightarrow\infty}\sup_{t\in\left[ 0,T\right] }\lvert{I^{\epsilon_n,11B}_t}\rvert\leq\delta_{\lvert\!\lvert{X}\rvert\!\rvert}(\nabla_\xr F,3\gamma)\sqrt{\left[ N,N\right]_T\left[ X,X\right]_T}.
\end{equation*}
On the other hand, since $N$ is continuous and hence uniformly continuous on $\left[ 0,T\right])$,     $\lvert{N_{s+\epsilon_n}-N_s}\rvert\leq\gamma$ $\forall{s}\in\left[ 0,T\right]$,  for $n$ large enough. Then
\begin{align*}
\sup_{t\in\left[ 0,T\right] }\lvert{I^{\epsilon_n,11A}_t}\rvert&\leq\sum_{i=1}^K\dfrac{1}{\epsilon_n}\int_{t_i-\epsilon_n}^{t_i}\int_0^1\rvert G_{\epsilon_{n}}(s,\lambda)\rvert{d\lambda}~\lvert{X_{s+\epsilon_n}-X_s}\rvert \lvert{N_{s+\epsilon_n}-N_s}\rvert{ds}\\
&\leq\gamma\times{K}\times{2}\lvert\!\lvert{X}\rvert\!\rvert\times{2}\sup_{s\in\left[ 0,T\right],~\lvert\!\lvert\xr\rvert\!\rvert\leq{\lvert\!\lvert{X}\rvert\!\rvert}}\nabla_\xr F_s(\xr).
\end{align*}
Hence, 
\begin{equation*}
\limsup_{n\rightarrow\infty}\sup_{t\in\left[ 0,T\right] }\lvert{I^{\epsilon_n ,11}_t}\rvert\leq\delta_{\lvert\!\lvert{X}\rvert\!\rvert}(\nabla_\xr F,3\gamma)\sqrt{\left[ N,N\right]_T\left[ X,X\right]_T}+4\gamma{K}\lvert\!\lvert{X}\rvert\!\rvert\sup_{s\in\left[ 0,T\right],~\lvert\!\lvert\xr\rvert\!\rvert\leq{\lvert\!\lvert{X}\rvert\!\rvert}}\nabla_\xr F_s(\xr)
\end{equation*}
which allows us to conclude that $I^{\epsilon_n,11}\underset{n\rightarrow{+\infty}}{\longrightarrow}0$ by 
arbitrariness of $\gamma>0$.  \\Since $\mathbb{P}(\mathcal{N})=1$,  we get $I^{\epsilon_n,11}\underset{n\rightarrow{+\infty}}{\longrightarrow}0$ uniformly a.s.~and thus the convergence holds u.c.p. Since it is true for all sequence $(\epsilon_n)_{n\in\N}$ that converges to 0, then $I^{\epsilon,11}\underset{\epsilon\rightarrow{0}}{\longrightarrow}0$ u.c.p., which concludes the proof.
\end{proof}

\nocite{*}
\bibliographystyle{plain}

\begin{thebibliography}{}

\end{thebibliography}


\begin{thebibliography}{10}

\bibitem{bandini2017weak}
Elena Bandini and Francesco Russo.
\newblock Weak dirichlet processes with jumps.
\newblock {\em Stochastic Processes and their Applications},
  127(12):4139--4189, 2017.

\bibitem{bandini2018special}
Elena Bandini, Francesco Russo.
\newblock Special weak dirichlet processes and bsdes driven by a random
  measure.
\newblock {\em Bernoulli}, 24(4A):2569--2609, 2018.

\bibitem{bouchard2021approximate}
Bruno Bouchard, Gr{\'e}goire Loeper, and Xiaolu Tan.
\newblock Approximate viscosity solutions of path-dependent pdes and dupire's
  vertical differentiability.
\newblock {\em arXiv preprint arXiv:2107.01956}, 2021.

\bibitem{bouchard2021ac}
Bruno Bouchard, Gr{\'e}goire Loeper, and Xiaolu Tan.
\newblock A $C^{0,1}$-functional It{\^o}'s formula and its applications in
  mathematical finance.
\newblock {\em arXiv preprint arXiv:2101.03759}, 2021.

\bibitem{bouchard2021quasi}
Bruno Bouchard and Xiaolu Tan.
\newblock A quasi-sure optional decomposition and super-hedging result on the
  skorokhod space.
\newblock {\em Finance and Stochastics}, pages 1--24, 2021.

\bibitem{BT19}
Bruno Bouchard and Xiaolu Tan.
\newblock Understanding the dual formulation for the hedging of path-dependent
  options with price impact.
\newblock {\em Annals of Applied Probability.}, to appear.

\bibitem{clark1970representation}
John~MC Clark.
\newblock The representation of functionals of brownian motion by stochastic
  integrals.
\newblock {\em The Annals of Mathematical Statistics}, pages 1282--1295, 1970.

\bibitem{cont2013functional}
Rama Cont, David-Antoine Fourni{\'e}, et~al.
\newblock Functional it{\^o} calculus and stochastic integral representation of
  martingales.
\newblock {\em The Annals of Probability}, 41(1):109--133, 2013.

\bibitem{coquet2006natural}
Fran{\c{c}}ois Coquet, Adam Jakubowski, Jean M{\'e}min, and Leszek
  S{\l}ominski.
\newblock Natural decomposition of processes and weak dirichlet processes.
\newblock In {\em In Memoriam Paul-Andr{\'e} Meyer}, pages 81--116. Springer,
  2006.

\bibitem{cosso2014regularization}
Andrea Cosso and Francesco Russo.
\newblock A regularization approach to functional it\^{o} calculus and
  strong-viscosity solutions to path-dependent pdes.
\newblock {\em arXiv preprint arXiv:1401.5034}, 2014.

\bibitem{cosso2019crandall}
Andrea Cosso and Francesco Russo.
\newblock Crandall-lions viscosity solutions for path-dependent pdes: The case
  of heat equation.
\newblock {\em arXiv preprint arXiv:1911.13095}, 2019.

\bibitem{Dupire2009FunctionalIC}
Bruno Dupire.
\newblock Functional it{\^o} calculus.
\newblock {\em Bloomberg: Frontiers (Topic)}, 2009.

\bibitem{gozzi2006weak}
Fausto Gozzi and Francesco Russo.
\newblock Weak dirichlet processes with a stochastic control perspective.
\newblock {\em Stochastic Processes and their Applications},
  116(11):1563--1583, 2006.

\bibitem{he2019semimartingale}
Sheng-wu He, Jia-gang Wang, and Jia-an Yan.
\newblock {\em Semimartingale theory and stochastic calculus}.
\newblock Routledge, 2019.

\bibitem{jacod2013limit}
Jean Jacod and Albert Shiryaev.
\newblock {\em Limit theorems for stochastic processes}, volume 288.
\newblock Springer Science \& Business Media, 2013.

\bibitem{russo1993forward}
Francesco Russo and Pierre Vallois.
\newblock Forward, backward and symmetric stochastic integration.
\newblock {\em Probability theory and related fields}, 97(3):403--421, 1993.

\bibitem{russo1995generalized}
Francesco Russo and Pierre Vallois.
\newblock The generalized covariation process and it{\^o} formula.
\newblock {\em Stochastic Processes and their applications}, 59(1):81--104,
  1995.

\bibitem{russo2007elements}
Francesco Russo and Pierre Vallois.
\newblock Elements of stochastic calculus via regularization.
\newblock In {\em S{\'e}minaire de Probabilit{\'e}s XL}, pages 147--185.
  Springer, 2007.

\bibitem{saporito2018functional}
Yuri~F Saporito.
\newblock The functional meyer--tanaka formula.
\newblock {\em Stochastics and Dynamics}, 18(04):1850030, 2018.

\end{thebibliography}

\end{document}